\documentclass[10pt]{amsart}
\title{The irreducible representations of 3-dimensional Sklyanin algebras}
\author{Kevin De Laet}
\address{Department of Mathematics, University of Antwerp \\ 
 Middelheimlaan 1, B-2020 Antwerp (Belgium) \\ {\tt kevin.delaet2@uantwerpen.be}}

\date{}
\usepackage{amsmath,amsfonts,array,amscd,amsthm,amssymb,stackrel,verbatim,wrapfig,subfig,float}
\usepackage{enumerate}
\usepackage[numbers, square, comma, sort&compress]{natbib}
\usepackage{youngtab}
\usepackage{url}
\usepackage{tikz}
\usepackage[all]{xy}
\usetikzlibrary{arrows,matrix}
\usetikzlibrary{shapes.geometric}
\usetikzlibrary{decorations.markings} % voor pijlen in midden van lijn

\tikzset{
  vertice/.style={circle,draw=black},
  decoration={markings,mark=at position 0.5 with {\arrow{>}}}
}

\newcommand{\wis}[1]{{\text{\em \usefont{OT1}{cmtt}{m}{n} #1}}}
\newcommand{\C}{\mathbb{C}}
\newcommand{\N}{\mathbb{N}}
\newcommand{\Z}{\mathbb{Z}}
\newcommand{\A}{\mathbb{A}}
\newcommand{\PP}{\mathbb{P}}

\newcommand{\V}{\mathbf{V}}
\theoremstyle{plain}
\newtheorem{theorem}{Theorem}%[section]

\newtheorem{proposition}{Proposition}

\newtheorem{conjecture}[theorem]{Conjecture}
\newtheorem{example}{Example}

\DeclareMathOperator{\diag}{diag}

\DeclareMathOperator{\tr}{tr}
\DeclareMathOperator{\Ann}{Ann}
\DeclareMathOperator{\END}{END}
\DeclareMathOperator{\qgr}{qgr}

\DeclareMathOperator{\coh}{coh}

\numberwithin{equation}{section}
\begin{document}
\begin{abstract}
In this article a complete description is given of the simple representations of a $3$-dimensional Sklyanin algebra associated to a torsion point. In order to determine these irreducible representations, a review is given of classical results regarding representation theory of graded rings with excellent homological and algebraic properties.
\end{abstract}
\maketitle

\section{Introduction}
The 3-dimensional Sklyanin algebras form the most important class of Artin-Schelter regular algebras of global dimension 3. These algebras are parametrized by the following geometric data: an elliptic curve $E$ and a point $\tau \in E$. In addition, they form a class of graded algebras with excellent ringtheoretical and homological properties, which was shown in the original papers by Artin, Tate and Van den Bergh, \citep{artin2007some} and \citep{artin1987graded}.
\par Recently, the representation theory of these algebras have gained attention, see for example \citep{bruyn2017superpotential}, \citep{de2015geometry}, \citep{walton2017poisson} and \citep{walton2012representation}.  In \citep{smith1994center} it was shown that a Sklyanin algebra is a finite module over its center if and only if $\tau$ is a torsion point of $E$, in which case the Sklyanin algebra is a maximal order over its center. Consequently, in the case of a torsion point, it is useful to describe the simple representations.
\par The point of this article is to explain the connection between the study of fat point modules, $\wis{PGL}_n \times \C^*$-stabilizers of simple representations and $\C^*$-families of (non-trivial) simple representations of a graded algebra $A$ with good properties. This connection is known by the experts, but it is not well-known by the generic mathematician working in noncommutative algebraic geometry.
As a consequence, we verify \citep[Conjecture 1.5]{walton2017poisson} using these methods.
\begin{conjecture}\citep[Conjecture 1.5]{walton2017poisson}
The simple representations of a 3-dimensional Sklyanin algebra with $\tau$ a torsion point of order $n$ with $(n,3)=3$ are of dimension $n$, $\frac{n}{3}$ or $1$. The simple representations of dimension $\frac{n}{3}$ form a $3:1$-cover over three lines intersecting in a unique point.
\end{conjecture}
\subsection{Statement}
The author hereby acknowledges that his own contribution is mainly combining previous results by Artin, Tate, Van den Bergh, Le Bruyn, Smith and others.
\subsection{Conventions}
In this article, the following notations and conventions are used.
\begin{itemize}
\item We will work over $\C$.
\item The element $\omega \in \C^*$ will be a primitive third root of unity.
\item The element $\rho \in \C^*$ will be a primitive $n$th root of unity for a specified $n \in \N$.
\item For a graded algebra $A$, $\wis{Proj}(A) = \qgr(A)$ is the category of finitely generated, graded $A$-modules, modulo torsion modules. Recall that there is an automorphism on $\wis{Proj}(A)$ by the shift functor: if $M = \oplus_{i \in \N} M_i \in \wis{Proj}(A)$, then $M[1]$ is the graded module with degree $k$-part isomorphic to $M_{k+1}$.
\item For $A$ any algebra, the set $\wis{irrep}A$ is the set of simple representations of $A$ up to equivalence. With $\wis{irrep}_k A$ for $k \in \N$ the subset of $k$-dimensional simple representations up to equivalence are denoted.
\item If $\mathcal{A}$ is a sheaf of coherent algebras over $X$ with $X$ a projective scheme, then $\wis{irrep}\mathcal{A}$ denotes the sheaf of sets of simple representations up to equivalence over $X$, that is, 
$$
(\wis{irrep}\mathcal{A})(U) = \wis{irrep} \mathcal{A}|_U.
$$
With $\wis{irrep}_k \mathcal{A}$ we denote the subsheaf of sets of simple $k$-dimensional representations up to equivalence over $X$.
\end{itemize}
\section{Noncommutative projective geometry}
\subsection{Cayley-Hamilton algebras}
In this section a review of the results of \citep{le1998generating}, \citep{lebruyn1995central} and \citep{le2007noncommutative} is discussed. We will assume that
\begin{itemize}
\item $$
A = \C \oplus A_1 \oplus A_2 \oplus \ldots
$$ is a graded algebra, finitely generated in degree one with quadratic relations,
\item $A$ is a domain,
\item $A$ has global dimension $d$ and $H_A(t) = (1-t)^{-d}$ for some $d \in \N$, and
\item $A$ is a finite module over its center $R = Z(A)$, which is a normal domain.
\end{itemize} 
The algebras of current interest, the Sklyanin algebras at torsion points, have all these properties by \citep[Corollary 1.3]{tate1996homological}.
\par Let $n\in \N$ be equal to $$n=\max \{m :\exists \phi:\xymatrix{A\ar@{->>}[r]
& \wis{M}_m(\C)} \text{ algebra morphism}\}.$$ Under these conditions, $A$ is a Cayley-Hamilton of degree $n$, that is, there exists an $R$-linear map $\tr: \xymatrix{A \ar[r]& R}$ such that
\begin{itemize}
\item $\tr(1)=n$,
\item $\tr(ab)=\tr(ba)$ for each $a,b, \in A$, and
\item $\chi_{n,a}(a)=0$ for each $a \in A$.
\end{itemize}
The element $\chi_{n,a}(a)$ is the $n$th Cayley-Hamilton polynomial of degree $n$, expressed in the sums of powers of $a$, as explained in \citep[Section 2.3]{procesi1987formal}.
\par A grading on an algebra $A$ is equivalent to an action of the torus group $\C^*$, by the rule
$$
a \in A_n \Leftrightarrow t \cdot a = t^k a \text{ for each } k \in \N, t \in \C^*, a \in A.
$$
As such, $\tr$ is also a degree-preserving map, that is, $\tr(A_n) \subset A_n$.
\begin{example}
Let $n \in \N$ be an integer strictly larger than $1$ and take the algebra
$$
A_\rho = \C_\rho[x,y] = \C\langle x,y \rangle/(xy-\rho yx),
$$
classically graded by $\deg(x)=\deg(y)=1$. Then $Z(A_\rho) = \C[x^n,y^n]$. The natural trace map on $A_\rho$ is the $\C$-linear extension of:
$$
\tr(x^ky^l) = \begin{cases}
n x^ky^l,\quad \text{if }(k,l) \in (n\N)^2,\\
0,\quad \text{otherwise}
\end{cases}
$$
The couple $(A_\rho,\tr)$ determines a Cayley-Hamilton algebra of degree $n$.
\label{ex:quantumunity}
\end{example}
To $A$ is associated the affine variety of trace-preserving $n$-dimensional representations, that is,
$$
X_A=\wis{trep}_n A := \{\phi:\xymatrix{A \ar[r]& \wis{M}_n(\C)} \text{ algebra morphism} :\tr\circ \phi = \phi \circ \tr\},
$$
with the trace on $\wis{M}_n(\C)$ being the usual trace map. The projective special linear group of degree $n$, $\wis{PGL}_n(\C)$ acts on this variety by conjugating matrices, such that $\wis{PGL}_n(\C)$-orbits correspond to isomorphism classes of $n$-dimensional representations. By a result of Artin \citep[Section 12.6]{artin1969azumaya}, one has
$$
X_A /\wis{PGL}_n(\C) \cong \wis{Spec}(R).
$$
In addition, as $A$ is graded, $\C^*$-acts on $X_A$. The only closed orbit, as in the commutative case, is the trivial representation $(A/A^+)^{\oplus n}$, which is also the only orbit with a non-trivial $\C^*$-stabilizer. From the fact that the actions of $\C^*$ and $\wis{PGL}_n(\C)$ on $X_A$ commute, it follows that $X_A$ is a $\wis{PGL}_n(\C) \times \C^*$-variety.
\par Let $M$ be a simple $n$-dimensional $A$-representation, then its $\wis{PGL}_n(\C)$-stabilizer is trivial by Schur's lemma. As mentioned before, the $\C^*$-stabilizer of $M$ is also trivial, but the $\wis{PGL}_n(\C) \times \C^*$-stabilizer doesn't have to be trivial. From \citep[lemma 4, theorem 2]{bocklandt2006local}, it follows that this stabilizer is always finite and is conjugated to the group generated by $(g_\zeta,\zeta) \in \wis{PGL}_n(\C)\times \C^*$, with $\zeta = \rho^{\frac{n}{e}}$ for some $e \in \N$ and
$$
g_\zeta = \diag(\underbrace{1,\ldots,1}_{m_0},\underbrace{\zeta,\ldots,\zeta}_{m_1},\ldots,\underbrace{\zeta^{e-1},\ldots,\zeta^{e-1}}_{m_{e-1}})
$$
\par If $M$ is a $k$-dimensional simple $A$-representation with $k < n$, one can look at $\wis{rep}_k A$ and calculate the $\wis{PGL}_k(\C) \times \C^*$-stabilizer.
\subsection{Graded matrix rings}
Let $M$ be a simple $A$-module of dimension $k$ with $k\leq n$ and let $\mathfrak{m} = \Ann_A(M)$.
Then we have a map
$$
\xymatrix{A_{\mathfrak{m}} \ar[r]^-\phi& \wis{M}_k(\C)}
$$
Let $$\mathfrak{m}^g = \max \{ I \subset  \mathfrak{m}: I \text{ homogeneous ideal} \}.$$
Then $\mathfrak{m}^g$ is a maximal homogeneous ideal. By a graded version of Artin-Wedderburn \citep[theorem I.5.8]{nastasescu2011graded}, there is an integer $e \in \N \setminus\{0\}$ and a $k$-tuple $(a_1,\ldots,a_k) \in \N^k$ with $0\leq a_1 \leq a_2 \leq \ldots \leq a_k < e$ such that
$$
A_{\mathfrak{m}^g}/\mathfrak{m}^g \cong \wis{M}_k(\C[t,t^{-1}])(a_1,a_2,\ldots,a_k)
$$
with $\deg(t) = e$. If $\deg(t) = 1$ then we will suppress the integers $a_1,\ldots,a_k$. This isomorphism is a graded isomorphism. Recall that for a $\Z$-graded ring $S$ and a $k$-tuple $(a_1,\ldots,a_k)\in \Z^k$ the matrix ring $\wis{M}_k(S)(a_1,a_2,\ldots,a_k)$ has degree $m$-part for $m \in \Z$
$$
\wis{M}_k(S)(a_1,a_2,\ldots,a_k)_m =
\begin{bmatrix} 
S_m & S_{m-a_1+a_2} & \hdots & S_{m-a_1+a_k} \\
S_{m-a_2+a_1} & S_m & \hdots & S_{m-a_2+a_k} \\
\vdots & \vdots & \ddots & \vdots \\
S_{m-a_k+a_1} & S_{m-a_k+a_2} & \hdots & S_m
\end{bmatrix}.	
$$
By construction, there is a commutative diagram
$$
\xymatrix{A_{\mathfrak{m}} \ar@{->>}[r]^-\phi& \wis{M}_k(\C)\\
A_{\mathfrak{m}^g} \ar@{->>}[r]^-\psi \ar[u]^-i& \wis{M}_k(\C[t,t^{-1}])(a_1,a_2,\ldots,a_k) \ar[u]_-{t \mapsto 1}}
$$
By \citep[Theorem 1]{le1998generating}, $\wis{M}_k(\C[t,t^{-1}])(a_1,a_2,\ldots,a_k)$ is generated in degree one over its center if and only if
$$
(a_1,a_2,\ldots,a_k) = (\underbrace{0,\ldots,0}_{m_0},\underbrace{1,\ldots,1}_{m_1},\ldots,\underbrace{e-1,\ldots,e-1}_{m_{e-1}})
$$
with $m_j >0$ and $\sum_{j=0}^{e-1} m_j = k$, which we will assume throughout this paper. By assumption, this means that if $x \in A_1$, then
$$
\psi(x)=\begin{bmatrix}
0& 0 &  \hdots & M(x)_0 t\\
M(x)_1 & 0 & \hdots & 0\\
\vdots & \ddots & \ddots& \vdots \\
0 &  \hdots & M(x)_{e-1} & 0
\end{bmatrix}
$$
with $M(x)_l \in \wis{M}_{m_l \times m_{l-1}}(\C)$, indices taken in $\Z_e$. But then, by construction, $\phi$ has as $\wis{PGL}_k(\C) \times \C^*$-stabilizer the cyclic group $\langle (g_\zeta,\zeta)\rangle$.
\begin{example}
Continuing example \ref{ex:quantumunity} with $\rho = -1$, then the center of $A=A_{-1}$ is $\C[x^2,y^2]$. Over $\mathfrak{m}=(x^2-1,y^2-1)$ lies a unique two-dimensional simple representation. Then $\mathfrak{m}^g = (x^2-y^2)$. As $Z(A_{\mathfrak{m}^g}) \subset \oplus_{k \in 2 \Z} (A_{\mathfrak{m}^g})_k$ and $A_{\mathfrak{m}^g}$ is generated over its center by degree one elements, it follows that
$$
A_{\mathfrak{m}^g}/\mathfrak{m}^g \cong \wis{M}_2(\C[t,t^{-1}])(0,1), \deg(t)=2.
$$
By construction, a $\C^*$-family of simple two-dimensional representations of $A$ is determined by
$$
\begin{bmatrix}
0 & t \\ 1 & 0
\end{bmatrix},
\begin{bmatrix}
0 & -t \\ 1 & 0
\end{bmatrix}\in \wis{M}_2(\C[t,t^{-1}])(0,1).
$$
For each $t \neq 0$ and associated algebra morphism $\phi_t$ obtained by specialising these two matrices, the $\wis{PGL}_2(\C)\times \C^*$-stabilizer is determined by
$$
\left\langle \begin{bmatrix} 1&0\\ 0 & -1 \end{bmatrix},-1\right\rangle.
$$
\label{example:q=-1}
\end{example}
\subsection{Periodic fat point modules}
Associated to 
$$
\xymatrix{A_{\mathfrak{m}^g}\ar@{->>}[r]^-\psi & \wis{M}_k(\C[t,t^{-1}])(\underbrace{0,\ldots,0}_{m_0},\underbrace{1,\ldots,1}_{m_1},\ldots,\underbrace{e-1,\ldots,e-1}_{m_{e-1}})}
$$
are $e$ $e$-periodic fat point modules (\citep[Section 6]{le1998generating}). Recall that a fat point module is a simple object of $\wis{Proj}(A)$. For $j \in \N$, let $V_j = \C^{m_l}$ if $j \equiv l \bmod \Z_e$. Let $M(\psi) = \oplus_{j \in \N} V_j$ and define an action of $A$ by extending
$$
x \in A_1, v \in M(\psi)_j=V_j: xv := M(x)_{j+1}v \text{ as an element of } M(\psi)_{j+1} = V_{j+1}.
$$
As a $\C[t]$-module, this is just $\oplus_{i=0}^{e-1} (\C[t])[i]^{\oplus m_i}$. Of course, this makes sense as $t$ acts faithfully on this module and so 
$$
\underline{\END}(\oplus_{i=0}^{e-1} (\C[t,t^{-1}][i])^{\oplus m_i}) = \wis{M}_k(\C[t,t^{-1}])(\underbrace{0,\ldots,0}_{m_0},\underbrace{1,\ldots,1}_{m_1},\ldots,\underbrace{e-1,\ldots,e-1}_{m_{e-1}}).
$$
By surjectivity of $\psi$, $M(\psi)$ is a fat point module of $A$. However, if $e>1$, then $M(\psi)[k] \not\cong M(\psi)$ for $1\leq k \leq e-1$, but $M(\psi)[e]\cong M(\psi)$ in $\wis{Proj}(A)$. By definition, $M(\psi)$ is an $e$-periodic fat point module of multiplicity $(m_0,m_1,\ldots,m_{e-1})$.
\par
The other way round, suppose that $M$ is an $e$-periodic fat point module of multiplicity $(m_0,m_1,\ldots,m_{e-1})$. Then to $M$ is associated a $\sum_{i=0}^{e-1} m_i=k$-dimensional representation $S$ by the following rule:
$$
x \in A_1, v \in M_j: \text{ if } xv = M(x)_{j+1}v \text{ as an element of } M_{j+1} = V_{j+1}
$$
then one associates the algebra map
$$
\psi_M(x)=\begin{bmatrix}
0& 0 &  \hdots & M(x)_0\\
M(x)_1 & 0 & \hdots & 0\\
\vdots & \ddots & \ddots& \vdots \\
0 &  \hdots & M(x)_{e-1} & 0
\end{bmatrix}.
$$ 
This algebra map corresponds to a simple representation of $A$, for if this was not simple, then this would imply that $M$ is not a fat point module.
\par In short, we have
\begin{theorem}
\citep[Proposition 3]{le1998generating}
There is a one-to-one correspondence between
\begin{itemize}
\item simple $k$-dimensional representations with $\wis{PGL}_k(\C)\times \C^*$-stabilizer generated by $(g_\zeta,\zeta)$ with $g_\zeta = \diag(\underbrace{1,\ldots,1}_{m_0},\underbrace{\zeta,\ldots,\zeta}_{m_1},\ldots,\underbrace{\zeta^{e-1},\ldots,\zeta^{e-1}}_{m_{e-1}})$ and
\item shift-equivalence classes of fat point modules of period $e$ and multiplicity
$(m_0,m_1,\ldots,m_{e-1})$.
\end{itemize}
\end{theorem}
If $m_0 = m_1 = \ldots =m_{e-1}=1$, then $M$ is called a point module. In fact, one has
\begin{proposition}
If $A$ has finite global dimension and has Hilbert series $(1-t)^{-d}$, then the Hilbert series of a fat point $M$ is 1-periodic, that is, $H_M(t) = m(1-t)^{-1}$ for some $m \in \N$.
\end{proposition}
\begin{proof}
Let $M$ be an $e$-periodic fat point. As $M$ has a finite projective resolution, it follows that
$$
\frac{f(t)}{(1-t)^l}=\frac{p(t)}{1-t^e} \text{ with } (1-t,f(t))=1, \deg(p(t))=e-1.
$$
But then it follows that $(1-t)^l$ divides $1-t^e$, which can only happen if $l=1$. In this case,
$$
f(t)(1+t+\ldots+t^{e-1}) = p(t),
$$
which can only happen if $f(t)=m$ is constant. But then the Hilbert series of $M$ is $m(1-t)^{-1}$.
\end{proof}
The integer $m$ of the proposition is called the multiplicity of $M$.
\begin{example}
Continuing example \ref{example:q=-1}, one sees that the algebra morphism 
$$
\xymatrix{A/\mathfrak{m}^g \ar@{->>}[r]& \wis{M}_2(\C[t,t^{-1}])(0,1)}, (x,y) \mapsto\left(\begin{bmatrix}
0 & t \\ 1 & 0
\end{bmatrix},
\begin{bmatrix}
0 & -t \\ 1 & 0
\end{bmatrix}\right)
$$
determines two point modules, the point module $P$ with associated infinite quiver
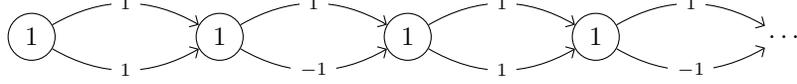
\begin{center}
\begin{tikzpicture}[scale = 2.5]
   \node[vertice] (a) at ( 0, 0) {$1$};
   \node[vertice] (b) at ( 1, 0) {$1$};
   \node[vertice] (c) at ( 2, 0) {$1$};
   \node[vertice] (d) at ( 3, 0) {$1$};
   \node (e) at ( 4, 0) {$\ldots$};
   \tikzset{every node/.style={fill=white}} 
   \path[->,font=\scriptsize,>=angle 90]
    (a) edge[out=30,in=150]node{$1$} (b)
    (a) edge[out=-30,in=-150]node{$1$} (b)
    (b) edge[out=30,in=150]node{$1$} (c)
    (b) edge[out=-30,in=-150]node{$-1$} (c)
    (c) edge[out=30,in=150]node{$1$} (d)
    (c) edge[out=-30,in=-150]node{$1$} (d)
    (d) edge[out=30,in=150]node{$1$} (e)
    (d) edge[out=-30,in=-150]node{$-1$} (e);
\end{tikzpicture}
\end{center}
and its shift $P[1]$.
\end{example}
\subsection{Sheaf of algebras $\mathcal{A}$}
In this subsection we follow \citep{lebruyn1995central}.
Let $g$ be a central element of $A$ of degree strictly larger than 0. Localizing $A$ at $g$ one gets the graded ring $\Lambda = A[g^{-1}]$ with center $\Sigma = R[g^{-1}]$ and taking the degree 0 part $\Lambda_0$, one gets an algebra over $\Sigma_0$ which is finitely generated as a $\Sigma_0$-module. Gluing these algebras, one gets a sheaf of orders $\mathcal{A}$, finite over $\mathcal{R}=\wis{Proj}(R)$. We have $\Lambda_0 \subset Q^{gr}(A)_0$, which is a division algebra, say of PI-degree $s$.
However, it is not always so that $Z(\Lambda_0) = \Sigma_0$.
\begin{example}
For any quantum algebra $A_\rho$ from example \ref{ex:quantumunity}, localizing at for example $x$, one finds that $\Lambda_0 = \C[\frac{y}{x}]$ and $\Sigma_0 = \C[\frac{y^n}{x^n}]$.
\end{example}
Consider now a fat point module $M$ which is not annihilated by $g$. Then $g$ acts faithfully on $M$ and $M[g^{-1}]$ is a graded $A[g^{-1}]$-module. The module $M[g^{-1}]$ becomes a gr-simple module (that is, it has no non-trivial graded submodules). By the equivalence of categories of Dade's Theorem \citep[Theorem 3.1.1]{nastasescu2011graded}, $N=M[g^{-1}]_0$ is a simple $\Lambda_0$-module.
\par Let $\mathcal{Z} = \wis{Spec}(Z(\mathcal{A}))$ be the central proj of $\mathcal{A}$. By a result of Artin \citep[Section 12.6]{artin1969azumaya}, $\mathcal{Z}$ parametrizes isomorphism classes of semi-simple trace-preserving $s$-dimensional representations of $\mathcal{A}$. There is a surjective map
$$
\xymatrix{\mathcal{Z} \ar@{->>}[r]^-\Phi & \mathcal{R}}.
$$
This map will be finite and therefore generically $e$-to-one for some $e \in \N$.
\begin{proposition}
The integer $e$ is equal to the least common multiple of the degrees of homogeneous elements of $R$.
\label{prop:e}
\end{proposition}
\begin{proof}
Localize $A$ at a finite number of generators of $R$, say $B = A[(x_1\cdots x_t)^{-1}]$ with $R = \C[x_1,\ldots,x_t]/(I)$ for $I$ some set of relations. Then any simple representation of dimension $n$ of $B$ comes from an algebra morphism
$$
\xymatrix{B_{\mathfrak{m}^g} \ar@{->>}[r]&\wis{M}_n(\C[t,t^{-1}])(\underbrace{0,\ldots,0}_m,\underbrace{1,\ldots,1}_m,\ldots,\underbrace{e-1,\ldots,e-1}_m)},
$$
for some maximal graded ideal $\mathfrak{m}^g$ of $B$, which we know from the previous discussions corresponds to a sum of $e$ distinct but shift-equivalent fat point modules of multiplicity $\frac{n}{e}$. As a fat point corresponds to a simple representation of $B_0$, we find that the PI-degree of $\mathcal{A}$ is $s=\frac{n}{e}$. As there are $e$ fat points lying over the graded prime ideal $\mathfrak{m}^g \cap R$, it follows that the map $$
\xymatrix{\mathcal{Z} \ar@{->>}[r]^-\Phi & \mathcal{R}}
$$
is generically $e$-to-one.
\end{proof}
The shift functor on $\wis{Proj}(A)$ induces an automorphism on the fat points of a fixed multiplicity. As we now know that fat points correspond to simple representations of $\mathcal{A}$, this shift functor also induces an automorphism $\alpha$ on $\mathcal{Z}$. As $\mathcal{Z}= \mathcal{A}/\wis{PGL}_n(\C)$ and $\wis{Spec}(R) = A/\wis{PGL}_n(\C)$, it follows that $\Phi$ is the quotient map by the group $G = \langle \alpha \rangle$.
\begin{theorem}
The following sets are in one-to-one correspondence:
\begin{itemize}
\item $e$-periodic multiplicity $m$-fat points of $A$,
\item $e$ $m$-dimensional simple representations of $\mathcal{A}$ such that the corresponding points in $\mathcal{Z}$ form an orbit under $\alpha$,
\item graded maximal ideals $\mathfrak{m}^g$ of $A$ such that
$$
A/\mathfrak{m}^g \cong \wis{M}_{em}(\C[t,t^{-1}])(\underbrace{0,\ldots,0}_m,\underbrace{1,\ldots,1}_m,\ldots,\underbrace{e-1,\ldots,e-1}_m), \deg(t)=e,
$$
\item simple $me$-dimensional representations with $\wis{PGL}_{me}(\C)\times \C^*$-stabilizer generated by a subgroup conjugated to $ \langle(g_\zeta,\zeta)\rangle$ with
$$
g_\zeta = \diag(\underbrace{1,\ldots,1}_{m},\underbrace{\zeta,\ldots,\zeta}_{m},\ldots,\underbrace{\zeta^{e-1},\ldots,\zeta^{e-1}}_{m}).
$$
\end{itemize}
\label{th:simporbits}
\end{theorem}
\begin{example}
Let
$$
A = \C\langle x,y,z\rangle/(yz-\omega zy,zx-\omega xz, xy-\omega yx)
$$
be an Artin-Schelter regular algebra of global dimension 3. Then the center of $A$ is isomorphic to $R \cong \C[u,v,w,d]/(uvw-d^3)$. However, the central proj of $\mathcal{A}$ is isomorphic to $\PP^2$, and in fact $\wis{Proj}(A) \cong \coh(\PP^2)$. In fact, $\wis{Proj}(R) = \PP^2/\Z_3$, with $\Z_3$ acting on $\PP^2 = \PP(\C \Z_3)$. The last fact can be proved using invariant theory: if $\PP^2 = \wis{Proj}(\C[x,y,z])$ with $\Z_3$-weights $0,1,2$, then
$$
\C[x,y,z]^{\Z_3} = \C[x,y^3,z^3,yz].
$$ 
As $\wis{Proj}(C) = \wis{Proj}(C^{(k)})$ for any affine, commutative ring $C$,
\begin{align*}
\wis{Proj}( \C[x,y^3,z^3,yz]) &\cong \wis{Proj}( \C[x,y^3,z^3,yz]^{(3)}) \\&= \wis{Proj}( \C[x^3,y^3,z^3,xyz]) \\&\cong \wis{Proj}(\C[u,v,w,d]/(uvw-d^3)).
\end{align*}
Consequenty, the map
$$
\xymatrix{\PP^2 \ar@{->>}[r]^-\Phi & \PP^2/\Z_3}
$$
is indeed just the quotient map by the automorphism $\alpha = \diag(1,\omega,\omega^2)$.
\end{example}
Armed with this knowledge, we can now tackle the problem of describing all the simple representations of $3$-dimensional Sklyanin algebras.
%\section{Sklyanin algebras}
%Fix $p$ an integer $\geq 3$. We follow \citep{smith1994point} in defining the $p$-dimensional Sklyanin algebra.
%\begin{definition}
%Let $\mathcal{L} = \mathcal{L}(pO)$ be a degree $p$ line bundle on an elliptic curve $E$, $\tau \in E$ and set $V = \dim H^0(E, \mathcal{L})$.
%Let $\Delta_\tau = \{(p,p+(n-2)\tau) \in E \times E: p \in E\}$. Let $M$ be the set of fixed points of $E \times E$ under the involution $\phi(x,y) = (y+2\tau,x-2\tau)$. We say that a divisor $D$ on $E \times E$ is allowable if $D$ is stable under $\phi$ and $M$ occurs in $D$ with even multiplicity. Then the $p$-dimensional Sklyanin algebra is associated to $E$ and $\tau$ is the algebra
%$$
%A(E,\tau) = T(V)/(R_\tau)
%$$
%with
%$$
%R_\tau = \{ f \in V \otimes V: f=0 \text{ or } (f)_0 = \Delta_\tau + D \text{ with } D \text{ allowable}\}.
%$$
%\end{definition}

\section{3-dimensional Sklyanin algebras}
In this section, set $p=3$. Let $E$ be an elliptic curve embedded in $\PP^2$ and a torsion point $\tau \in E$ of order $n>1$. The associated 3-dimensional Sklyanin algebra $A=A(E,\tau)$ is a finite module over its center by \citep[Theorem 1.2]{tate1996homological}. In this section we will give a complete description of the simple representations of $A$.
\par The first thing we have to notice is: if $g$ is the unique central element of degree 3, then all fat points of $A/(g)$ are point modules, which are parametrized by $E$. This follows from the fact that $ A/(g) = \mathcal{O}_\tau(E)$ is the twisted coordinate of $E$ associated to $\tau$ \citep[Theorem 6.8]{artin2007some}. In addition, the shift functor $[1]$ on $E \subset \wis{Proj}(\mathcal{O}_\tau(E))$ is addition with $\tau$, so that each point module is $n$-periodic. From this we deduce using Theorem \ref{th:simporbits}:
\begin{proposition}
Each non-trivial simple representation of $\mathcal{O}_\tau(E)$ is $n$-dimensional with $\wis{PGL}_{n}(\C)\times \C^*$-stabilizer isomorphic to $\Z_n$.
\end{proposition}
This fact is independent of the conditions $(n,3)=1$ or $(n,3)=3$.
\subsection{$(n,3)=1$}
The first thing one has to determine is the center of $A$.
\begin{theorem}\citep[Theorem 4.8]{smith1994center}
The center of $A$ is isomorphic to
$$
R=\C[u,v,w,g]/(\Psi(u,v,w)-g^n),$$
with $\deg(u)= \deg(v)=\deg(w) = n, \deg(g) = 3$ 
and $\Psi(u,v,w)$ a homogeneous polynomial of degree $3n$ which embeds the elliptic curve $E'=E/\langle \tau \rangle$ in $\PP^2_{[u:v:w]}$.
Consequently, the central proj is $\PP^2 \cong \wis{Proj}(R)$.
\end{theorem}
For the sake of completeness, a review of the discussion in \citep{de2015geometry} will be given. By \citep[Theorem 3.4]{artin1992geometry}, all the fat points that are not annihilated by $g$ are of multiplicity $n$, which also follows from \citep[Theorem 7.3]{artin1991modules} and the above discussion regarding simple representations of the sheaf $\mathcal{A}$ of dimension $n$ and multiplicity $n$ fat point modules.
\begin{theorem}\citep[Lemma 1, Theorem 4]{de2015geometry}
The simple representations of $A$ come in three types, let $S$ be a simple $A$-module with annihilator $\mathfrak{m}=\Ann(S)$.
\begin{itemize}
\item If $g \not\in \mathfrak{m}$ then $S$ is a $n$-dimensional representation of $A$ with trivial $\wis{PGL}_n(\C)\times \C^*$-stabilizer. The $\C^*$-orbit of $S$ corresponds to a unique 1-periodic fat point module of multiplicity $n$, which in turn corresponds to a unique $n$-dimensional representation of $\mathcal{A}$. Consequently, one has
$$
A_{\mathfrak{m}^g}/\mathfrak{m}^g \cong \wis{M}_n(\C[t,t^{-1}]) \text{ with } \deg(t)=1.
$$
\item If $g \in \mathfrak{m}$ but $\dim(S) \neq 1$ then $S$ is $n$-dimensional with $\wis{PGL}_n(\C)\times \C^*$-stabilizer isomorphic to $\Z_n$. The $\C^*$-orbit of $S$ corresponds to $n$ shift-equivalent point modules, who in turn correspond to $n$ 1-dimensional representations of $\mathcal{A}$. Consequently, on has 
$$
A_{\mathfrak{m}^g}/\mathfrak{m}^g \cong \wis{M}_n(\C[t,t^{-1}])(0,1,\ldots,n-1) \text{ with } \deg(t)=n.
$$
\item If $\dim(S)=1$ then $S$ is the trivial representation $A/A^+$.
\end{itemize}
\end{theorem}
Consequently, in the affine case, we have
\begin{align*}
&\xymatrix{\wis{irrep}_n A \ar@{->>}[r]^-{1:1}& \wis{irrep} R\setminus \V(u,v,w,g)},\\
&\xymatrix{\wis{irrep}_1 A \ar@{->>}[r]^-{1:1}& \wis{irrep} \V(u,v,w,g)}\\
\end{align*}
while in the projective case
\begin{align*}
&\xymatrix{\wis{irrep}_n \mathcal{A} \ar@{->>}[r]^-{1:1}& \wis{Proj}(R)\setminus E' = \PP^2 \setminus E'},\\
&\xymatrix{\wis{irrep}_1 \mathcal{A}=E \ar@{->>}[r]^-{n:1}&  E'}.
\end{align*}

\subsection{$(n,3)=3$}
If $(n,3) = 3$, then there will be a difference between $\mathcal{R}=\wis{Proj}(R)$ and $\mathcal{Z}$. Let $E' = E/\langle 3 \tau \rangle$ and $E'' = E'/\langle \tau \rangle$. We will assume that $n \neq 3$.
\begin{theorem}\citep[Theorem 4.8]{smith1994center}
The center of $A$ is isomorphic to 
$$
R=\C[u,v,w,g]/(\Psi(u,v,w)+3z^2 g^{\frac{n}{3}} + 3z g^{\frac{2n}{3}} + g^n)$$
 with $\deg(u)=\deg(v)=\deg(w)=n, \deg(g)=3$. The polynomial $\Psi(u,v,w)$ of degree $3n$ corresponds to $E''$ embedded in $\PP^2_{[u:v:w]}$ and $z$ is a linear polynomial in $u$,$v$,$w$ that vanishes on three inflection points of $E''$.
\label{th:center}
\end{theorem}
By Proposition \ref{prop:e}, the map
$$
\xymatrix{\mathcal{Z} \ar@{->>}[r]^-\Phi & \mathcal{R}}
$$
will be a quotient map by an automorphism $\alpha$ of order 3, which is the greatest common divisor of the degrees of the homogeneous elements of $R$. By \citep[Theorem 4.7]{smith1994center}, $\mathcal{Z} \cong \PP^2$ and $E' \subset \mathcal{Z}$.  Let $s =\frac{n}{3}$, then again by \citep[Theorem 3.4]{artin1992geometry}, all the fat points not annihilated by $g$ are of multiplicity $s$.
\par As $E'$ is the ramification locus of $\mathcal{A}$ (that is, the closed subset of $\mathcal{Z}$ where $\mathcal{A}$ is not Azumaya), it follows that $\alpha$ has to induce an automorphism of order three on $E'$ without fix points. This can only be if $\alpha$ is conjugated to $\langle\diag(1,\omega,\omega^2)\rangle \subset \wis{PGL}_3(\C)$. But then, $\alpha$ has three fixed points, which correspond to three fat points of multiplicity $s$ which are 1-periodic. They in turn correspond to three $\C^*$-families of simple representations of dimension $s$, with trivial $\wis{PGL}_s(\C) \times \C^*$-stabilizer.
\begin{theorem}
The simple representations of $A$ come in four types. Let $S$ be a simple $A$-module with annihilator $\mathfrak{m}=\Ann(S)$.
\begin{itemize}
\item If $g \not\in \mathfrak{m}$ and $S$ is a quotient of a fat point $F$ that is not fixed by the shift functor, then $S$ is $n$-dimensional with $\wis{PGL}_n(\C) \times \C^*$-stabilizer isomorphic to $\Z_3$. The fat point $F$ is 3-periodic and has multiplicity $s$. $F$ and its shifts $F[1]$ and $F[2]$ correspond to three different $s$-dimensional representations of $\mathcal{A}$. Consequently, one has
$$
A_{\mathfrak{m}^g}/\mathfrak{m}^g \cong \wis{M}_n(\C[t,t^{-1}])(\underbrace{0,\ldots,0}_{s},\underbrace{1,\ldots,1}_{s},\underbrace{2,\ldots,2}_{s}) \text{ with } \deg(t)=3.
$$
\item If $g \not\in \mathfrak{m}$ and $S$ is a quotient of a fat point $F$ that is fixed by the shift functor, then $S$ is $s$-dimensional with trivial $\wis{PGL}_s(\C) \times \C^*$-stabilizer. The multiplicity of $F$ is $s$ and $F$ corresponds to a unique $s$-dimensional representation of $\mathcal{A}$. Consequently, one has
$$
A_{\mathfrak{m}^g}/\mathfrak{m}^g \cong \wis{M}_s(\C[t,t^{-1}])\text{ with } \deg(t)=1.
$$
\item If $g \in \mathfrak{m}$ and $S$ is not one-dimensional, then $S$ is $n$-dimensional with $\wis{PGL}_n(\C)\times \C^*$-stabilizer isomorphic to $\Z_n$. The $\C^*$-orbit of $S$ corresponds to $n$ shift-equivalent point modules, who in turn correspond to $n$ 1-dimensional representations of $\mathcal{A}$. Consequently, on has 
$$
A_{\mathfrak{m}^g}/\mathfrak{m}^g \cong \wis{M}_n(\C[t,t^{-1}])(0,1,\ldots,n-1) \text{ with } \deg(t)=n.
$$
\item If $\dim(S)=1$ then $S$ is the trivial representation $A/A^+$.
\end{itemize}
\end{theorem}
\par Regarding the center $R$ of $A$, we can now give a different description.
\begin{proposition}
The center $R$ is isomorphic to $\C[u,v,w,d,g]/(uvw-d^3, g^s - l)$ with $l$ a linear form in $u$, $v$, $w$ and $d$ such that $R/(g)$ is the homogeneous coordinate ring of $E''$, generated in degree $n$.
\end{proposition}
\begin{proof}
By \citep[Theorem 3.7, Proposition 4.2]{smith1994center}, there are three normal elements $u_0,v_0,w_0$ of degree $s$ such that $\wis{Proj}(\C[u_0,v_0,w_0]) = \PP^2 = \mathcal{Z}$. By the fact that the shift functor corresponds to $\diag(1,\omega,\omega^2)$ for an open subset of fat points, it follows that each irreducible representation of dimension $n$ restricted to $\C[u_0,v_0,w_0]$ is determined by the matrices for $(\lambda,\mu,\eta) \in \A^3$.
$$
\left(\begin{bmatrix}
0_s & 0_s & \lambda 1_s\\
\omega \lambda 1_s & 0_s & 0_s\\
0_s & \omega^2 \lambda 1_s & 0_s
\end{bmatrix},
\begin{bmatrix}
0_s & 0_s & \mu 1_s\\
\omega \mu 1_s & 0_s & 0_s\\
0_s & \omega^2 \mu 1_s & 0_s
\end{bmatrix},
\begin{bmatrix}
0_s & 0_s & \eta 1_s\\
\omega \eta 1_s & 0_s & 0_s\\
0_s & \omega^2 \eta 1_s & 0_s
\end{bmatrix}\right)
$$
for an open subset of $\A^3$. But then, $u=u_0^3$, $v=v_0^3$, $w=w_0^3$ and $d=u_0v_0w_0$ are sent to scalar matrices for an open subset of $\wis{Spec}(R)$, so they are central in $A$. In addition, they are linearly independent, so they generate $R^{(n)}$ by Theorem \ref{th:center}. But then $g^s = l$ for some linear form in $u$, $v$, $w$ and $d$. The fact that
$$
\V(uvw-d^3,g^s-l)\cap \V(g) = E''
$$
follows from the fact that the center of $A/(g)$ is $\mathcal{O}(E/\langle \tau \rangle) = \mathcal{O}(E'')$. 
\end{proof}
In the affine case, we now find
\begin{align*}
&\xymatrix{\wis{irrep}_n A \ar@{->>}[r]^-{1:1}& \wis{irrep} R \setminus \V(uv,vw,uw)},\\
&\xymatrix{\wis{irrep}_s A \ar@{->>}[r]^-{3:1}& \V(uv,uw,uw)\setminus \V(u,v,w,g)},\\
&\xymatrix{\wis{irrep}_1 A \ar@{->>}[r]^-{1:1}& \V(u,v,w,g)},
\end{align*}
which verifies \citep[Conjecture 1.5]{walton2017poisson}, while in the projective case we find
\begin{align*}
&\xymatrix{\wis{irrep}_s \mathcal{A} \ar@{->>}[r]^-{1:1}& \wis{irrep} \mathcal{Z} \setminus E'},\\
&\xymatrix{\wis{irrep}_1 \mathcal{A}=E \ar@{->>}[r]^-{s:1}&  E'},\\
&\xymatrix{(\wis{irrep} \mathcal{Z}=\PP^2) \setminus \V(u_0v_0,v_0w_0,u_0w_0) \ar@{->>}[r]^-{3:1}& \mathcal{R}\setminus \V(uv,vw,wu,d)},\\
&\xymatrix{\V(u_0v_0,v_0w_0,u_0w_0) \ar@{->>}[r]^-{1:1}& \V(uv,vw,wu,d)}.
\end{align*}
\bibliographystyle{abbrv}
\bibliography{representationssklyaninalgebras}
\nocite{*}

\end{document}